\documentclass[11pt]{article}

\usepackage{amsmath,amsthm,amssymb, amsfonts,amscd}
\usepackage{mathtools}
\usepackage{bbm}
\usepackage{mathrsfs}
\usepackage{enumitem} 
\usepackage[utf8]{inputenc}

\usepackage[dvipsnames]{xcolor}

\newtheorem{theorem}{Theorem}[section]
\newtheorem{lemma}{Lemma}[section]

\theoremstyle{definition}
\newtheorem{remark}{Remark}[section]

\newcommand{\Z}{\mathbb{Z}}

\newcommand{\N}{\mathbb{N}}

\newcommand{\E}{\mathbb{E}}
\newcommand{\prob}{\mathbb{P}}
\newcommand{\1}{\mathbbm{1}}

\numberwithin{equation}{section}
\allowdisplaybreaks[1]

\begin{document}

\title{{Moderate deviations for the self-normalized random walk in random scenery}}

\author{
\normalsize{\textsc{Tal Peretz}\footnote{Technion - Israel Institute of Technology. 
E-mail: tal.peretz@campus.technion.ac.il}\ \ \ }}

\maketitle

\begin{abstract}
Let $G$ be an infinite connected graph with vertex set $V$. Let $\{ S_n: n  \in \N_0 \}$ be the simple random walk on $G$ and let $\{ \xi(v) : v \in V\}$ be a collection of i.i.d. random variables which are independent of the random walk. Define the random walk in random scenery as $T_n = \sum_{k=0}^n \xi(S_k)$, and the normalization variables $V_n = (\sum_{k=0}^n \xi^2(S_k))^{1/2}$ and $L_{n,2}= (\sum_{v \in V}\ell^2_n(v))^{1/2}$. For $G=\Z^d$ and $G=\mathbb T_d$, the $d$-ary tree, we provide large deviations results for the self-normalized process $T_n \sqrt{n}/(L_{n,2}V_n)$ under only finite moment assumptions on the scenery.
\newline
\newline
\emph{Keywords and phrases.} Moderate deviations; Self-normalized partial sums; Random walk in random scenery; Local times
\newline
MSC 2010 \emph{subject classifications.} Primary 60F10; Secondary 60G50, 60K37.
\end{abstract}

\section{Introduction}
Let $ G$ be an infinite connected graph with vertex set $ V$ and let $\{S_n: n \in \N_0 \}$ be the simple random walk on $G$ started at a distinguished vertex  $o \in V$. Denote the law and expectation of this walk by $P$ and $E$. Let $\{\xi(v): v \in V \}$ be independent copies of a random variable $\xi$, which we denote as the scenery. Denote the law of the scenery by $\prob$ and the expectation with respect to this law by $\E$. We will always assume $\E\xi = 0$ and $\sigma^2 = \E\xi^2>0$. The random walk in random scenery (RWRS) is the process $\{T_n: n \in \N_0\}$ defined by
\begin{align*}
    T_n = \sum_{k=0}^n \xi(S_k) = \sum_{v \in V}\ell_n(v)\xi(v),
\end{align*}
where $\ell_n(v) = \sum_{k=0}^n \1{\{S_k=v\}}$ is the local time of $v$ at time $n$. This process was introduced for the case $G=\Z^d$ by Kesten and Spitzer \cite{Kesten1979}, and by Borodin \cite{Borodin1982} independently and at the same time in order to introduce new scaling and self-similar laws. However for $d\geq 3$, when the random walk by time $n$ visits most points a constant amount of times, \cite{Kesten1979} showed under appropriate assumptions on the distribution of $\xi$ that $T_n/\sqrt{n}$ converges in distribution to a Gaussian random variable. More recently, large and moderate deviations of $T_n$ have been studied for $G= \Z^d$ in \cite{Asselahnote,Asselah,Feng,FLEISCHMANN2008,gantert2006deviations}.  Fleischmann, Mort\"{e}rs and Wachtel \cite{FLEISCHMANN2008} proved a moderate deviations principle (MDP) for $d \geq 3$. Assuming Cram\'{e}r's condition, i.e. that there exists $\theta>0$ such that $\E e^{\theta \vert \xi \vert}<\infty$, they showed that 
\begin{align}\label{RWRSmoderatedeviations}
    \lim_{n\rightarrow \infty}y_n^{-2} \log\prob \otimes P(T_n/\sqrt{n} \geq y_n) = -\frac{1}{2 \sigma^2(2G(0)-1)}
\end{align}
for $y_n = o(n^{1/6})$, where $G(\cdot)$ is the Green's function of the random walk. We write $a_n = o(b_n)$ or $a_n \ll b_n$ if $\lim_{n \to \infty }a_n/b_n =0$ for positive sequences $a_n$ and $b_n$. In contrast with moderate deviations of sums of i.i.d. random variables, in \cite{Asselah} it was shown that this regime is maximal under Cram\'{e}r's condition. That is, more assumptions need to be made on the scenery in order to get moderate deviations when $y_n$ grows faster than $n^{1/6}$. 

There has been a recent interest in proving large deviations for sums of i.i.d. random variables under minimal moment assumptions. It is well understood that if one replaces the normalization constant by self-normalization, this is possible. In  \cite{shao1997self}, Shao provided self-normalized large and moderate deviations for the partial sum of i.i.d. random variables, while only making assumptions on the second moment. In \cite{Feng}, Feng, Shao and Zeitouni extended this framework to RWRS by proving a Cr\'{a}mer type moderate deviations. Define 
\begin{align*}
V_n^2 = \sum_{k=0}^n \xi^2(S_k) = \sum_{v \in V}\ell_n(v)\xi^2(v)
\end{align*}
and
\begin{align*}
L^2_{n,2} =  \sum_{v \in V}\ell^2_n(v).
\end{align*}
For the simple random walk on $\Z^d$ for $d\geq 3$, it is known (see \cite{FLEISCHMANN2008, Kesten1979}) that 
\begin{align*}
T_n/\sqrt{n} &\xrightarrow{d} N(0,\sigma^2(2G(0)-1)), \enspace L^2_{n,2}/n \xrightarrow{p} 2G(0)-1
 \enspace \text{and} \enspace V_n^2/n \xrightarrow{p} \sigma^2,
\end{align*}
so that we have the self-normalized central limit theorem 
\begin{align*}
\frac{T_n \sqrt{n}}{V_n L_{n,2}} \xrightarrow{d} N(0,1).
\end{align*}
 In \cite{Feng}, they proved that if $d \geq 4$ and there exists $\alpha>0$ and $c_\alpha>0$ such that $\prob(\xi \geq t) \leq 2 e^{-c_\alpha t^\alpha}$ for $t>0$, then
\begin{align}\label{SN-MDP-Zd}
 \prob \otimes P\left(\frac{T_n \sqrt{n}}{V_n L_{n,2}} \geq x\right)\sim 1-\Phi(x)
\end{align}
uniformly for $x \in [0,O(n^\tau)]$ and any $0< \tau < \alpha/(6 \alpha + 4)$. Here $\Phi(\cdot)$ is the standard normal distribution function, and we write $a_n \sim b_n$ if $\lim_{n \to \infty} a_n/b_n = 1$ for positive sequences $a_n$ and $b_n$. By self-normalizing, \cite{Feng} was able to achieve a MDP while only assuming sub-exponential tails on the scenery, which contrasts \eqref{RWRSmoderatedeviations}. However, for a suitable range of $y_n$, one would expect
\begin{align}\label{quadraticMD}
-\infty<\liminf_{n\to \infty}y_n^{-2} \log \prob \otimes P\left(\frac{T_n \sqrt{n}}{V_n L_{n,2}} \geq y_n\right) \leq \limsup_{n\to \infty}y_n^{-2} \log \prob \otimes P\left(\frac{T_n \sqrt{n}}{V_n L_{n,2}} \geq y_n\right) < 0
\end{align}
under only finite moment assumptions on the scenery, i.e. $\E \vert \xi \vert^\kappa < \infty$ for some fixed $\kappa>0$. In the same paper, the authors showed that if $\xi$ has the probability density function $\frac{1}{2}\alpha (1+ \vert t \vert)^{-1-\alpha}$ for some $\alpha>0$ and $(\log n)^{1/2} \ll y_n \ll n^{1/2}$, then
\begin{align}\label{Zdlowerbound}
\liminf_{n \to \infty} y_n^{-2d/(d+2)} (\log n)^{-2/(d+2)} \log \prob \otimes P\left(\frac{T_n \sqrt{n}}{V_n L_{n,2}} \geq y_n\right) > - \infty.
\end{align}
Since $y_n^2 \gg y_n^{2d/(d+2)} (\log n)^{2/(d+2)} $, this lower bound shows that moderate deviations of the form \eqref{quadraticMD} does not hold when $y_n \gg (\log n)^{1/2}$ if one is only to assume finite moment conditions on $\xi$. 

The lower bound is achieved by considering the event in which the random walk occupies the ball $\{z \in \Z^d: \vert z \vert \leq R \}$ for time $y_n^2$, and taking each of the scenery values inside the ball to be of size $O(n)$. The probability of the former event is roughly $\exp(-y_n^2/R^2)$, and for the latter event  $n^{-R^d} = \exp(-R^d \log n)$. Since the random walk and the scenery are independent of one other, we get the lower bound $\exp(-y_n^2/R^2 - R^d \log n)$ which, optimized over $R$, yields \eqref{Zdlowerbound}. From this example, we observe that deviations of the self-normalized RWRS depend on the interplay between the scenery and the random walk. Furthermore, we see that this process is sensitive to the correlation of the local times due to the heavy tails of the scenery.

Motivated by this phenomenon, in this paper we study self-normalized moderate deviations for graphs other than the lattice. We expect that for graphs where the simple random walk has weakly-dependent local times, moderate deviations of the form \eqref{quadraticMD} are attainable. A natural candidate for such graphs are trees, since regeneration epochs of the random walk on the tree have exponential tails (see Section \ref{timesection} for definitions). Lastly, we also provide an upper bound that complements \eqref{Zdlowerbound}. 
\section{Main results}
Let $\mathbb T_d$ be the $d$-ary tree rooted at $o$, i.e. $\deg(v) = d+1$ for $v \neq o$ and $\deg(o)=d$. In this paper we will always assume $d \geq 2$. Since the simple random walk on $\mathbb T_d$ is transient, it follows that for $G=\mathbb T_d$ we have $T_n/\sqrt{n}$ converges weakly to a normal random variable. The following theorem provides precise asymptotics for the self-normalized RWRS on $\mathbb T_d$ while only making finite moment assumptions on the scenery.
\begin{theorem} \label{theorem}
	Let $G=\mathbb T_d$ and let $y_n$ be a positive sequence such that $y_n \to \infty$.
\begin{enumerate}
	\item Suppose that $\E \xi^{4} < \infty$ and $y_n = o(n^{1/6})$. We then have
	\begin{align}\label{upperbound}
	\limsup_{n \to \infty}y_n^{-2}\log \prob \otimes P \left(\frac{T_n \sqrt{n}}{V_n L_{n,2}} \geq y_n\right) \leq -\frac{c_d}{2},
	\end{align}
	where $c_d$ is a positive constant independent of $n$, and $c_d \uparrow 1$ as $d \to \infty$. 
	\item Suppose $\E \xi^6 < \infty$ and $y_n = o(n^{1/6})$. We then have
	\begin{align} \label{lowerbound}
	\liminf_{n \to \infty} y_n^{-2}\log \prob \otimes P \left(\frac{T_n \sqrt{n}}{V_n L_{n,2}} \geq y_n\right) \geq -\frac{1}{2}.
	\end{align}
\end{enumerate}
\end{theorem}
In light of \eqref{Zdlowerbound}, the following theorem provides a sharp upper bound for the self-normalized RWRS on the lattice when only assuming finite moment conditions on the scenery.
\begin{theorem}\label{Zdupperbound}
Let $G = \Z^d$ for $d \geq 3$ and let $y_n$ be a positive sequence such that $(\log n)^{1/2} \ll y_n \ll n^{1/6}$. If $\E \xi^4 < \infty$, then
\begin{align}
\limsup_{n \to \infty} y_n^{-2d/(d+2)} (\log n)^{-2/(d+2)} \log \prob \otimes P\left(\frac{T_n \sqrt{n}}{V_n L_{n,2}} \geq y_n\right) <0.
\end{align} 
\end{theorem}
\begin{remark}
For Theorems \ref{theorem} and \ref{Zdupperbound} we assume that $y_n \ll n^{1/6}$. In this regime, the deviation comes from the moderate deviation of the scenery. When $y_n \gg n^{1/6}$, the deviation comes from large deviations for the local time statistics $L_{n,2}$ and $L_{n,3}$, see Section \ref{timesection} for definitions.
\end{remark}

The rest of the paper is organized as follows. In Section \ref{timesection}, we study concentration inequalities for different local time statistics that will aid us in the proof of Theorem 
\ref{theorem}. Our main tool will be the regeneration structure of the random walk on the tree. In Sections \ref{upper-section} and \ref{lower-section}, we prove the upper and lower bound of Theorem \ref{theorem}. In Section \ref{Zdlocaltimesection} we review the necessary concentration inequalities for local time in the lattice, and in Section \ref{Zdupperboundproof} we prove Theorem \ref{Zdupperbound}.

\section{Local time for $\mathbb T_d$} \label{timesection}
For the rest of the paper will write $P_v$ and $E_v$ when the random walk is conditioned on starting at $v\in V$. We will also denote $a_n \asymp b_n$ if $0<\liminf_{n\to \infty} a_n/b_n \leq \limsup_{n\to\infty} a_n/b_n
< \infty$ for positive sequences $a_n$ and $b_n$. 
\subsection{Regeneration times}
Our proofs for the concentration of local times will utilize the regenerative structure of the random walk on the tree. For $v \in V$, denote the level of $v$ by $\vert v \vert$, which is the length of the unique geodesic between $v$ and $o$. We define regeneration times as
\begin{align*}
\tau_1 = \inf \{n \in \N : \vert S_n \vert  \neq  \vert S_k \vert \: \text{for all} \: k < n \: \text{and} \: \vert S_k \vert \neq  \vert S_{n-1} \vert \: \text{for all} \: k >n\}
\end{align*}
and for $j \in \N$
\begin{align*}
\tau_{j+1} = \inf \{n > \tau_j : \vert S_n \vert  \neq  \vert S_k \vert \: \text{for all} \: k < n \: \text{and} \: \vert S_k \vert \neq  \vert S_{n-1} \vert \: \text{for all} \: k >n \}.
\end{align*}
Regeneration times were studied in \cite{Dembo2002,Lyons1996} for biased random walks on Galton-Watson trees. Rerunning the same proofs, we can conclude that the simple random walk on $\mathbb T_d$ has infinitely many regeneration times such that $ \{\tau_{j+1} - \tau_j\}_{j \geq 1}$ is an i.i.d. sequence which is independent of $\tau_1$, and that $\tau_1$ and $\tau_2 - \tau_1$ have exponential moments. The following lemma shows that as $d$ tends to infinity, the tails of the regeneration times become lighter. 
\begin{lemma} \label{reg-moment}
	Define 
\begin{align*}
s_d = \sup \{ \lambda \geq 0: E e^{\lambda \tau_1} < \infty \: \:  \text{and} \: \: E e^{\lambda (\tau_2 - \tau_1)} < \infty \}.
\end{align*}	
We then have $s_d \geq \frac{1}{3} \log((d+1)/3)+\frac{1}{3}-\frac{1}{d+1}$.
\end{lemma}
\begin{proof}
Since we are only concerned with the levels of $\mathbb T_d$, we can consider our random walk as a Markov chain on $\N_0$ starting at $0$, with transition probabilities $p_{0,1} = 1$, $p_{j,j+1} =d/(d+1)$ and $p_{j,j-1} = 1/(d+1)$ for $j \in \N$. We claim that
	\begin{align*}
	\{ \tau_1 = k\} \subset \{ \text{RW took at least $\lfloor k/3 \rfloor$ steps backwards by time $k$}\}.
	\end{align*}
Observe that $\vert S_{\tau_1} \vert \leq \lfloor \tau_1/3 \rfloor$. This is because for every $m \in \N_0$ such that $m<\vert S_{\tau_1} \vert$, the random walk must visit $m$ at least twice. Now suppose by contradiction the random walk took less than $\lfloor k/3 \rfloor$ steps backwards, which means the random walk took more than $k-\lfloor k/3 \rfloor$ steps forward. This implies $\vert S_{\tau_1} \vert> \lfloor k/3 \rfloor$, which is a contradiction. Applying the Chernoff inequality to the binomial random variable $Z$ with parameters $k \in \N_0$ and $1/(d+1)$, it follows that for all $x \geq k/(d+1)$ 
\begin{align*}
P(Z> x) \leq \exp\left(-x \log\left(\frac{x}{k/(d+1)}\right)+x-\frac{k}{d+1}\right).
\end{align*}
By our inclusion and the above inequality, we have
\begin{align*}
P(\tau_1 = k) \leq P(Z \geq k/3) \leq \exp \left(-\frac{k}{3} \log((d+1)/3)+\frac{k}{3}-\frac{k}{d+1} \right).
\end{align*}
Hence when $\lambda < \frac{1}{3} \log((d+1)/3)+\frac{1}{3}-\frac{1}{d+1}$, which is strictly positive when $d>2$, we have $E e^{\lambda \tau_1} < \infty$. We are left to bound $P(\tau_2 - \tau_1 = k)$. Rerunning the proof of Lemma 4.3 in \cite{Dembo2002}, we conclude that there exists a constant $C>0$ such that for every $k$ we have $P(\tau_2-\tau_1=k) \leq C \cdot P(\tau_1=k)$. This finishes the proof.
\end{proof}
\subsection{Concentration inequalities}
We introduce the various local time statistics of the simple random walk on $\mathbb T_d$ used throughout the remainder of this paper. The one we are most interested in is the size of the level sets of the local time
\begin{align*}
    L_n(t) = \vert \{v \in V : \ell_n(v) > t  \}\vert.
\end{align*}
Let $\mathcal R_n = \{S_0,\ldots, S_{n-1} \}$ be the range of the random walk at time $n$. Setting $\lambda_d = s_d /2$, we define
\begin{align*}
\mathcal L^{(n)} = \mathcal L =\left \{v \in \mathcal R_n : \ell_n(v) < \frac{4}{\lambda_d}\log n \right\}
\end{align*}
to be the set of vertices with small local time. Another statistic that appears throughout the proof is 
\begin{align*}
\mathscr L_n = \sum_{v \in \mathcal L^c} \ell_n(v).
\end{align*}
For $q \in \N$, denote the $q$-fold self-intersection local time by
\begin{align*}
    L_{n,q} = \left(\sum_{v \in V}\ell_n^q(v) \right)^{1/q}.
\end{align*}
The $q$-fold self-intersection local times often appear in the study of RWRS because they quantify the number of times the random walk visits the same sites, see \cite{Asselah} for a discussion for the case $\Z^d$. Throughout the paper, we will frequently use the fact that 
\begin{align} \label{selfintersect}
L_{n,q}^q \geq n.
\end{align}
 Lastly, denote the maximum of the local times by 
\begin{align*}
    L_{n,\infty} = \max_{v \in \mathcal R_n} \ell_n(v).
\end{align*}
We will denote $\theta_k = \tau_{k}-\tau_{k-1}$ for $k \geq 1$ and $\theta_1 = \tau_1$ to be the regeneration epochs. Our main ingredient for deriving concentration inequalities for the local time will be the existence of regeneration epochs, and that they have exponential moments.

\begin{lemma}
Suppose $u \geq 1$ and $t \geq 1$. There exists a constant $M>0$ independent of $u$ and $t$ such that for $\beta \in (0,\lambda_d/2]$ we have
\begin{align}\label{localmoment-lemma}
    E\exp\left(\beta\sum_{k=1}^n  \theta_k \cdot \1\{\theta_k >t\} \right)\leq\exp(Mn\exp(-\beta t/2)),
\end{align}
which implies
\begin{align} \label{levelset-lemma} 
    P(L_n(t) \geq u) \leq \exp(Mn\exp(-\beta t/2)- \beta tu ).
\end{align}
\end{lemma}

\begin{proof}
We begin by making a few observations. Firstly, at the $k$th regeneration epoch $\theta_k$, there are at most $\lfloor \theta_k/t \rfloor$ vertices $v$ satisfying $\ell_n(v)> t$.  Secondly, any visited site is visited during a single regeneration epoch. Since by time $n$ there are at most $n$ regeneration times, we get 
\begin{align*}
    L_n(t) \leq \sum_{k=1}^n \left\lfloor \theta_k /t \right\rfloor \cdot \1\{\theta_k >t\}.
\end{align*}
Combining this inequality with the Chebyshev inequality, we have for any $\beta \in (0,\lambda_d]$ 
\begin{align*}
    P(L_n(t) \geq u) &\leq P\left(\sum_{k=1}^n \theta_k \cdot \1\{\theta_k >t\} \geq ut\right) \\
     &\leq \exp(-\beta ut) \prod_{k=1}^nE\exp \left(\beta  \theta_k\cdot \1\{\theta_k >t\} \right),
\end{align*}
where the last inequality uses the fact that the regeneration epochs are independent. We are left to bound the exponential moment. Again by the Chebyshev inequality, we have $P(\theta_k > t) \leq M \exp(-\beta t)$ for some positive constant $M$. Assuming $\beta \in (0,\lambda_d/2]$, we have
\begin{align*}
 \prod_{k=1}^nE\exp \left(\beta  \theta_k \cdot \1\{\theta_k >t\} \right) &\leq \prod_{k=1}^n\bigg(E[\exp \left(\beta  \theta_k \right)\1\{\theta_k>t\}]+ 1 \bigg) \\
 &\leq \prod_{k=1}^n\bigg((E\exp \left(2\beta  \theta_k \right))^{1/2} P(\theta_k > t)^{1/2}+ 1 \bigg) \\
 &\leq \exp(Mn\exp(-\beta t/2)).
\end{align*}
\end{proof}
\begin{lemma}\label{localsum-lemma}
There exists an $M > 0 $ such that for any $u \geq 1$
\begin{align*}
    P\left(\mathscr L_n \geq u\right) \leq M\exp(- \lambda_d u/2).
\end{align*}
\end{lemma}
\begin{proof}
We have the inequality $\mathscr L_n \leq \sum_{k=1}^n  \theta_k \cdot \1\{\theta_k >\frac{4}{\lambda_d} \log n\}$. The Chebyshev inequality combined with an application of \eqref{localmoment-lemma} finishes the proof. 
\end{proof}
\begin{lemma}\label{max-lemma}
Suppose $x> 1$. Then there exists a constant $c_1>0$ independent of $n$ and $x$ such that
\begin{align*}
    P(L_{n,\infty} \geq x) \leq n \exp(-c_1 (x-1)). 
\end{align*}
\end{lemma}
\begin{proof}
For $v \in V$, define $T^+_v = \inf \{n \geq 1 : S_n = v \}$ to be the return time of $v$ and let $p_v = P_v(T^+_v = + \infty)$ be the escape probability starting at $v$. Observe that the probability of escaping from $o$ is the same as the probability of escaping from $v$ conditioned on the event of $\{  \vert S_1 \vert > \vert v \vert\}$. Therefore,
\begin{align*}
    p_o = P_v(T^+_v = +\infty \vert \vert S_1 \vert > \vert v \vert ) =\frac{P_v(T_v^+= +\infty, \vert S_1 \vert > \vert v \vert )}{P_v(\vert S_1 \vert > \vert v \vert)}\leq \frac{P_v(T_v^+= +\infty)}{P_v(\vert S_1 \vert > \vert v \vert)}= \frac{d+1}{d}p_v.
\end{align*}
We now have the uniform lower bound $p_v \geq p_o d/(d+1).$ By the strong Markov property, we have 
\begin{align*}
    P_v(\ell_n(v) \geq x) \leq (1-p_v)^{x-1} \leq \left(1-\frac{d}{d+1}p_o\right)^{x-1}.
\end{align*}
Using the same proof as Lemma 18 in \cite{FLEISCHMANN2008}, we have
\begin{align*}
P(L_{n,\infty} \geq x) \leq \sum_{v \in V} P(\ell_n(v) \geq x) &= \sum_{v \in V}\sum_{k=1}^n P(T_v^+=k)P_v( \ell_{n-k}(v) \geq x) \\
&\leq \sum_{v \in V}P_v( \ell_{n}(v) \geq x)\sum_{k=1}^n P(T_v^+=k) \\
&\leq \left(1-\frac{d}{d+1}p_o\right)^{x-1}\sum_{v \in V} P(T_v^+\leq n).
\end{align*}
We finish by observing $\sum_{v \in V} P(T_v^+\leq n)  \leq \sum_{v \in V} \sum_{k=1}^n P(S_k = v)  = n$.
\end{proof}
\begin{lemma}\label{silt-lemma} 
Fix an integer $q \geq 2$. There exist positive constants $B_q$ and $c_q$ such that
\begin{align*}
    \limsup_{n \to \infty } n^{-1/q}\log P\bigg(L_{n,q}^q \geq B_qn\bigg) \leq - c_q. 
\end{align*}
\end{lemma}
\begin{proof}
Fix $B_q>0$  and define the events $E = \{L_{n,q}^q \geq B_q n\}$ and $F = \{L_{n,\infty} > n^{1/q} \}$.  By Lemma \ref{max-lemma}, we are left to bound $P(E \cap F^c)$, which we will do by bounding the probability that the level sets of the local time are large. Define the sets 
\begin{align*}
    \mathcal D_k = \left \{ v \in V : 2^{k-1} < \ell_n(v) \leq 2^{k}\right\}
\end{align*}
for $k=0, \ldots,  \lceil \log_2(n^{1/q}) \rceil = K_2$, and define the events
\begin{align*}
    D_k &= \left\{ \vert \mathcal D_k \vert > \frac{8 M n}{\lambda_d}\frac{e^{-\lambda_d 2^{k-1}/8}}{2^{k-1}} \right\} \enspace &&\text{for} \enspace k=0, \ldots, K_1 =  \lfloor \log_2(\frac{8}{\lambda_d}\log n) \rfloor \\
   D_k &= \left\{ \vert \mathcal D_k \vert > \frac{n^{1/q}
}{2^{ k-1}}\right\} \enspace &&\text{for} \enspace k=K_1 + 1 ,\ldots, K_2.
\end{align*}
For $k=1,\ldots, K_1$, we apply \eqref{levelset-lemma} for $\beta = \lambda_d/4$ and get
\begin{align*}
    P(D_k) \leq P\left(L_n(2^{k-1}) \geq  \frac{8Mn}{\lambda_d}\frac{e^{-\lambda_d 2^{k-1}/8}}{2^{k-1}}\right) \leq \exp\left(- M n e^{-\lambda_d 2^{k-1}/8}\right) \leq \exp(-M n^{1/2}).
\end{align*}
For $v \in \mathcal D_k$ for $k=K_1 + 1,\ldots, K_2$ we have $\ell_n(v) \geq 4\log n/\lambda_d$, so we can apply \eqref{levelset-lemma} with parameter $\beta = \lambda_d/2$ to get
\begin{align*}
    P(D_k) \leq P\left(L_n(2^{ k - 1}) \geq \frac{n^{1/q}}{2^{k-1}}\right) \leq  \exp\left(M-\frac{\lambda_d}{2} n^{1/q}\right).
\end{align*}
By the union bound it follows that there exists $c_q>0$ independent of $n$ such that
\begin{align*}
    \limsup_{n \to \infty } n^{-1/q}\log P\left( \cup_{k=0}^{K_2} D_k \right)\leq - c_q. 
\end{align*}
It is left to show $E \cap F^c  \subset  \left (\cup_{k=0}^{K_2} D_k \right) \cap F^c$. On the event $F^c$, we have $V =  \cup_{k=0}^{K_2} \mathcal D_k $. Therefore, on the event $\left (\cup_{k=0}^{K_2} D_k \right)^c \cap F^c$, we get 
\begin{align*}
    L_{n,q}^q \leq \sum_{k=0}^{K_2} 2^{qk} \vert \mathcal D_k \vert \leq \frac{8Mn}{\lambda_d}\sum_{k=0}^{K_1} 2^{(q-1)k+1}e^{- \lambda_d 2^{k-1}/8} + n^{1/q} \sum_{k=K_1 + 1}^{K_2} 2^{(q-1)k+1} \leq B_q n
\end{align*}
for some constant $B_q$. 
\end{proof}
\section{Proof of the upper bound  \eqref{upperbound}} \label{upper-section}
Our strategy will be to decompose $T_n/V_n$ into summands according to the size of the local times and scenery values. The probability that $T_n/V_n$ is large will be rewritten as the probability that each summand is large. The probability that the summand over small local time and scenery value is large essentially reduces to the regime of bounded i.i.d. random variables, see Lemma \ref{good-sum}. The probability that the summands over large local time and large scenery value is large will be bounded by the events $\mathscr L_n$ and $\vert \mathcal E^c \vert$ are large, see definitions below.

Without loss of generality, assume $\E \xi^2 = 1$. 
Before we continue, we introduce notation for the rest of the proof. Define the sets
\begin{align*}
    \mathcal L^{(n)} =\mathcal L= \left\{v \in \mathcal R_n : \ell_n(v) < \frac{4}{\lambda_d}\log n \right\} \qquad \text{and} \qquad
      \mathcal E^{(n)}= \mathcal E = \left\{ v\in \mathcal R_n : \xi(v)< \frac{\sqrt{ n}}{y_n \log^2 n}  \right\},
\end{align*}
as well as the partial sums
\begin{align*}
    T_{n,1} &= \sum_{v\in \mathcal L \cap \mathcal E} \ell_n(v)\xi(v), & V^2_{n,1} &= \sum_{v\in \mathcal L \cap \mathcal E} \ell_n(v)\xi^2(v), \\
    T_{n,2} &= \sum_{v\in \mathcal L \cap \mathcal E^c} \ell_n(v)\xi(v), & V^2_{n,2} &= \sum_{v\in \mathcal L \cap \mathcal E^c} \ell_n(v)\xi^2(v)\\
    T_{n,3} &= \sum_{v\in \mathcal L^c } \ell_n(v)\xi(v), & V^2_{n,3} &= \sum_{v\in \mathcal L^c } \ell_n(v)\xi^2(v),
\end{align*}
so that $T_n = T_{n,1} + T_{n,2} + T_{n,3}$ and $V^2_n = V^2_{n,1} + V^2_{n,2}+V^2_{n,3}$. Before we prove the upper bound, we need a few auxiliary results.
 \begin{lemma} \label{good-sum}
Suppose $1\ll y_n \ll n^{1/6}$ and $\E \xi^4 < \infty$. We then have
\begin{align*}
   \limsup_{n\to \infty}y_n^{-2}\log \prob \otimes P\left(\frac{T_{n,1} \sqrt{n}}{V_{n,1} L_{n,2}} \geq y_n \right) \leq -\frac{1}{2}.
\end{align*}
 \end{lemma}
 \begin{lemma}\label{scenery-lemma}
Suppose that $\E \vert \xi^{m}\vert< \infty$ for some fixed $m \in \N$. For $x >0$, we have
\begin{align*}
    \prob\otimes P(\vert \mathcal E^c \vert \geq x) \leq \left(e \E \vert \xi^m\vert \frac{y_n^m \log^{2m}(n)}{xn^{m/2-1}}\right)^x.
\end{align*}
\end{lemma}
\begin{proof}[Proof of \eqref{upperbound}]
Fix $\epsilon \in (0,1)$. We have
\begin{align*}
\prob \otimes P\left( \frac{T_n \sqrt{n}}{V_n L_{n,2}} \geq y_n  \right) &\leq \prob \otimes P\left( \frac{T_{n,1} \sqrt{n}}{V_{n,1} L_{n,2}} \geq \epsilon y_n  \right) \\
  &+ \prob \otimes P  \left( \frac{T_{n,2} \sqrt{n}}{V_{n,2} L_{n,2}} \geq (1-\epsilon) y_n/2  \right) \\
  &+  \prob \otimes P\left( \frac{T_{n,3} \sqrt{n}}{V_{n,3} L_{n,2}} \geq (1-\epsilon) y_n/2 \right) \\
  &\eqqcolon P_1 + P_2 + P_3.
\end{align*}
To bound $P_1$, we simply apply Lemma \ref{good-sum}
\begin{align*}
P_1 =  \prob \otimes P\left( \frac{T_{n,1} \sqrt{n}}{V_{n,1} L_{n,2}} \geq \epsilon y_n  \right) \leq \exp\left(-\frac{\epsilon^2}{2}y_n^2(1-o(1))\right).
\end{align*}
To bound $P_2$, we observe that by \eqref{selfintersect} and the Cauchy-Schwartz inequality,
\begin{align*}
\frac{T_{n,2} \sqrt{n}}{V_{n,2} L_{n,2}} \leq \frac{T_{n,2}}{V_{n,2}} =  \frac{\sum_{v \in \mathcal L \cap \mathcal E^c } \ell_n(v) \xi(v)}{\sqrt{\sum_{v \in \mathcal L \cap \mathcal E^c} \ell_n(v) \xi^2(v)}}\leq \sqrt{\sum_{v \in \mathcal L \cap \mathcal E^c} \ell_n(v)  } \leq \sqrt{\frac{4}{\lambda_d} \log n \vert \mathcal E^c \vert}.
\end{align*}
Combining the above inequality and Lemma \ref{scenery-lemma} yields
\begin{align*}
P_2  = \prob \left( \frac{T_{n,2} \sqrt{n}}{V_{n,2} L_{n,2}} \geq (1-\epsilon) y_n/2  \right)  &\leq \prob \left(\vert \mathcal E^c\vert \geq (1-\epsilon)^2\frac{\lambda_d}{16} \frac{y_n^2}{\log n}\right) \\
&\leq  \left( \frac{e \E \xi^4}{(1-\epsilon)^2\lambda_d /16} \frac{y_n^2 \log^9 n}{n}\right)^{(1-\epsilon)^2\frac{\lambda_d}{16}\frac{y_n^2}{\log n}} \\
&= \exp \left( -(1-\epsilon)^2\frac{\lambda_d}{16} y_n^2\left(\frac{\log(n/y_n^2)}{\log n}-o(1)\right) \right)\\
&\leq \exp \left( -(1-\epsilon)^2\frac{\lambda_d}{24} y_n^2(1-o(1)) \right).
\end{align*}
To bound $P_3$, we again use \eqref{selfintersect} and the Cauchy-Schwartz inequality
\begin{align*}
\frac{T_{n,3} \sqrt{n}}{V_{n,3} L_{n,2}} \leq  \frac{T_{n,3}}{V_{n,3}} =  \frac{\sum_{v \in  \mathcal L^c } \ell_n(v) \xi(v)}{\sqrt{\sum_{v \in \mathcal L^c} \ell_n(v)  \xi^2(v)}} \leq \sqrt{\sum_{v \in \mathcal L^c}  \ell_n(v)}= \sqrt{\mathscr L_n}.
\end{align*}
Applying this as well as Lemma \ref{localsum-lemma} gives us
\begin{align*}
P_3  = P \left( \frac{T_{n,3} \sqrt{n}}{V_{n,3} L_{n,2}} \geq (1-\epsilon) y_n/2  \right)  \leq P \left(\mathscr L_n \geq (1-\epsilon)^2 y_n^2/4\right) \leq M\exp \left( -(1-\epsilon)^2\frac{\lambda_d}{8} y_n^2\right).
\end{align*}
We conclude that for $d\geq 2$
\begin{align*}
\limsup_{n \to \infty}y_n^{-2}\log \prob \otimes P\left( \frac{T_n \sqrt{n}}{V_n L_{n,2}} \geq y_n \right) < -\min \left\{\frac{\epsilon^2}{2},(1-\epsilon)^2\frac{\lambda_d}{24},(1-\epsilon)^2\frac{\lambda_d}{8} \right\}.
\end{align*}
Recall from Lemma \ref{reg-moment} that $\lambda_d \uparrow \infty$ as $d \to \infty$. Hence for $d$ large enough we can let $\epsilon = 1-(24/\lambda_d)^{1/2}$, and get
\begin{align*}
\limsup_{n \to \infty}y_n^{-2}\log \prob \otimes P\left( \frac{T_n \sqrt{n}}{V_n L_{n,2}} \geq y_n \right) \leq - \frac{1}{2}( 1-(24/\lambda_d)^{1/2})^2.
\end{align*}
Letting $c_d = ( 1-(24/\lambda_d)^{1/2})^2$ finishes the proof.

\end{proof}
We are left to prove Lemmas \ref{good-sum} and \ref{scenery-lemma}. The first lemma is a self-normalized moderate deviation result; since our sum is over vertices with small local time, the proof is very similar to the i.i.d. regime. The proof of the second lemma is a straightforward Chernoff inequality.
\begin{proof}[Proof of Lemma \ref{good-sum}]
 Recall the local time statistics $\mathscr L_n = \sum_{v \in \mathcal L^c} \ell_n(v)$, $L_{n,2}^2 = \sum_{v \in V}\ell_n^2(v)$ and $L_{n,3}^3 = \sum_{v \in V}\ell_n^3(v)$. We begin by defining the atypical event
\begin{align*}
A = \left\{L_{n,2}^2> B_2 n \right\}\cup \left\{L_{n,3}^3> B_3 n \right\} \cup \{\mathscr L_n \geq y_n^2\log n \}.
\end{align*}
Since $y_n =o( n^{1/6})$, by Lemmas \ref{silt-lemma} and \ref{localsum-lemma} we have
\begin{align*}
    \limsup_{n \to \infty }y_n^{-2}\log P(A ) =-\infty.
\end{align*}
Hence we are left to bound
\begin{align*}
   \prob \otimes P \left( \left\{\frac{T_{n,1} \sqrt{n}}{V_{n,1} L_{n,2}} \geq y_n \right\} \cap A^c \right).
\end{align*}
Fix $\delta \in (0,1)$. We decompose our probability with respect to the size of $V_{n,1}$:
\begin{align*}
\prob \otimes P \left(\left\{\frac{T_{n,1} \sqrt{n}}{V_{n,1} L_{n,2}} \geq y_n\right\} \cap A^c\right) &\leq \prob\otimes P  \left( \{ V_{n,1}^2\leq \delta n \} \cap A^c\right)  + \prob\otimes P  \left(\left\{\frac{T_{n,1}}{ L_{n,2}} \geq \delta^{1/2} y_n \right\} \cap A^c\right) \\
&\eqqcolon I_1 + I_2. 
\end{align*}
We first show that $I_1$ is negligible. By the Chebyshev inequality, for any $\kappa>0$ we have
\begin{align*}
	\prob  (V_{n,1}^2 < \delta n ) & \leq \exp(\kappa  \delta n) \E [ \exp(-\kappa V_{n,1}^2)]  \\
	& = \exp(\kappa \delta n) \prod_{v \in \mathcal L} \E[ \exp(-\kappa \ell_n(v) \xi^2(v)\1\{v \in \mathcal E\})].
\end{align*}
By monotone convergence, we have $\E[\xi^2(v)\1\{v \in \mathcal E\}] = 1-o(1)$. Using this, and that $e^{-x} \leq 1- x+x^2/2$ for $x\geq 0$, we have
\begin{align*}
    \E [\exp(-\kappa \ell_n(v) \xi^2(v)\1\{v \in \mathcal E\})] &\leq \E[1-\kappa \ell_n(v) \xi^2(v)\1_{v \in \mathcal E} + {\kappa^2}\ell_n^2(v) \xi^4(v)\1\{v \in \mathcal E\}/2] \\
    &\leq 1 -\kappa \ell_n(v)(1-o(1))+ \E [\xi^4]\kappa^2\ell_n^2(v)/2  \\
    &\leq \exp( -\kappa \ell_n(v)(1-o(1)) + \E [\xi^4]\kappa^2\ell_n^2(v)/2).
 \end{align*}
On the event $A^c$, we have $\sum_{v \in \mathcal L}\ell_n(v) \geq n - y_n^2\log n $. We then get
\begin{align*}
  I_1  &\leq\exp( \kappa \delta n) E\left[\exp\left(-\kappa \sum_{v \in \mathcal L} \ell_n(v)(1-o(1))+ \E [\xi^4] \kappa^2 \sum_{v \in \mathcal L} \ell_n^2(v)/2\right)\1_{A^c}\right] \\
    & \leq  \exp(-\kappa  (1-\delta) n(1-o(1)) + B_2 \E [\xi^4] \kappa^2 n/2).
\end{align*}
Optimizing over $\kappa$ gives the bound
\begin{align*}
I_1 \leq \exp\left(-\frac{(1-\delta)^2 }{2B_2\E [\xi^4]} n(1-o(1)) \right),
\end{align*}  
and since $y_n =o(n^{1/6})$, we have for ever $\delta \in (0,1)$
\begin{align*}
    \limsup_{n \to \infty} y_n^{-2} \log I_1= -\infty.
\end{align*}
The rest of the proof is left to bound $I_2$. Applying Chebyshev's inequality with $\delta^{1/2}y_n$, we have
\begin{align*} 
I_2 &\leq \exp(-\delta y_n^2 )\cdot\E\otimes E\left[\exp(\delta^{1/2} y_n T_{n,1} /L_{n,2})\1_{A^c}\right ]  \\
&= \exp(-\delta y_n^2 )\cdot E\left[\prod_{v \in \mathcal L }\E\left[\exp\left(\delta^{1/2}y_n\frac{ \ell_n (v)}{L_{n,2}}  \xi(v)\1\{v \in \mathcal E\}\right)\right] \1_{A^c} \right].
\end{align*}
Since $L_{n,2}\geq  n^{1/2}$, for $v \in \mathcal L $ we have
\begin{align*}
    \delta^{1/2}y_n\frac{ \ell_n (v)}{L_{n,2}}  \xi(v)\1\{v \in \mathcal E\} \leq \frac{4 \delta^{1/2} }{\lambda _d\log n}.
\end{align*}
There exists a constant $C>0$ such that for any $x\leq 1$,
\begin{align*}
    e^x \leq 1+ x + x^2/2+C\vert x\vert^3.
\end{align*}
By this, and that $\E [\xi(v) \1\{v \in \mathcal E\}] \leq 0$, we have for large enough $n$
\begin{align*}
 I_2 \leq & \exp(-\delta y_n^2 ) E\left[ \prod_{v \in \mathcal L} \left(1+ \frac{\delta}{2}\frac{\ell^2_n(v)}{L^2_{n,2}} y_n^2 + O(1) \frac{\ell^3_n(v)}{L_{n,2}^3} y_n^3\right) \1_{A^c}\right]\\
  \leq & \exp(-\delta y_n^2 ) E\left[ \prod_{v \in \mathcal L} \exp \left(\frac{\delta}{2}\frac{\ell^2_n(v)}{L^2_{n,2}} y_n^2 + O(1) \frac{\ell^3_n(v)}{L_{n,2}^3} y_n^3\right) \1_{A^c}\right]\\
  \leq & \exp(-\delta y_n^2/2 ) E\left[ \exp \left(O(1) \frac{L_{n,3}^3}{L_{n,2}^3} y_n^3\right) \1_{A^c}\right].
 \end{align*}
Since $L^3_{n,3} \asymp n$ and $L^2_{n,2} \asymp n$ on the event $A^c$, as well as that $y_n = o(n^{1/6})$, we have the bound
 \begin{align*}
 \limsup_{n\to \infty} y_n^{-2} \log I_2 \leq -\frac{\delta}{2}.
  \end{align*}
Taking $\delta \uparrow 1$ finishes the proof.
\end{proof}
\begin{proof}[Proof of Lemma \ref{scenery-lemma}]
Observe that $$\left\{\1\left\{\vert\xi(v)\vert>\frac{\sqrt{n}}{y_n \log^2n}\right\}: v \in \mathcal R_n\right\}$$ are i.i.d. random variables with respect to $\prob$, and that by Markov's inequality we have 
\begin{align*}
    \prob\left(\vert \xi(v)\vert \geq \frac{\sqrt{n}}{y_n \log^2n}\right) \leq \E \vert\xi^m\vert\frac{y_n^m \log^{2m}(n)}{n^{m/2}}.
\end{align*}
Applying the Chernoff inequality to the binomial random variable $Z$ with parameters $n \in \N_0$ and $p \in [0,1]$, it follows that for all $x >0$ 
\begin{align*}
P(Z > x) \leq \left(\frac{enp}{x} \right)^{x}.
\end{align*}
Using this, and that $\vert \mathcal R_n \vert \leq n$ $P$-a.s, we have
\begin{align*}
    \prob \otimes P\left(\vert \mathcal E^c \vert \geq x \right)&= E\left[\prob\left(\sum_{v \in \mathcal R_n}\1\left\{\vert \xi(v)\vert>\frac{\sqrt{n}}{y_n \log^2n}\right\}>x \right) \right] \\
    & \leq E\left[\left(\frac{e\vert \mathcal R_n \vert \prob\left(\vert \xi(v)\vert \geq \frac{\sqrt{n}}{y_n \log^2n}\right)}{x} \right)^{x} \right] \\
    & \leq \left(\frac{e n \prob\left(\vert \xi(v)\vert \geq \frac{\sqrt{n}}{y_n \log^2n}\right)}{x} \right)^{x}  \\
    &\leq \left(e \E \vert \xi^m\vert \frac{y_n^m \log^{2m}(n)}{xn^{m/2-1}}\right)^x.
\end{align*}
\end{proof}
\section{Proof of the lower bound \eqref{lowerbound}} \label{lower-section}
The following proof is a straightforward application of the techniques used in the proof of  Theorem 2.2 in \cite{Feng}. Without loss of generality, assume $\E\xi^2 = 1$. Let $x,y$ and $b$ be positive numbers. By the the Cauchy-Schwartz inequality we have
\begin{align*}
    xy \leq \frac{1}{2}\left(\frac{x^2}{b} + y^2 b\right).
\end{align*}
Letting $b = y_nL_{n,2}/n$, $x = y_n L_{n,2}/\sqrt{n}$ and $y = V_n$, we get
\begin{align*}
    &\prob \otimes P\left(T_n \geq  V_n\frac{y_nL_{n,2}}{\sqrt{n}}\right) \\
    \geq& \prob \otimes P\left(T_n \geq \frac{1}{2b}\left(b^2V^2_n + y_n^2\frac{L^2_{n,2}}{n}\right)\right) \\
    =& \prob \otimes P\left(\sum_{v \in V} \ell_n(v)( 2b \xi(v)-b^2 \xi^2(v)) \geq y_n^2\frac{L^2_{n,2}}{n}\right) \\
    \geq& E\left[\prob\left(\sum_{v \in V} \ell_n(v)( 2b \xi(v)-b^2 \xi^2(v)) \geq y_n^2\frac{L^2_{n,2}}{n}\right)\1 \{  L_{n,2}^2 \leq B_2 n, L_{n,3}^3 \leq B_3 n\}\right].
\end{align*}
We are left to bound the inner probability, for which we will use Theorem 2 from \cite{nagaev2002lower}. Assume the random walk is fixed such that
\begin{align}\label{assumptions}
    L_{n,2}^2 \leq B_2 n \qquad \text{and} \qquad L_{n,3}^3 \leq B_3 n.
\end{align}
Defining $\eta(v) = 2b \xi(v) - b^2\xi^2(v)$, we have
\begin{align*}
    \prob\left(\sum_{v \in V} \ell_n(v)( 2b \xi(v)-b^2 \xi^2(v)) \geq y_n^2\frac{L^2_{n,2}}{n}\right) = \prob\left(\sum_{v \in V} \ell_n(v)(\eta(v)-\E\eta(v)) \geq 2y_n^2\frac{L^2_{n,2}}{n}\right).
\end{align*}
Define
\begin{align*}
    M_n^2 &= \sum_{v \in V} \ell_n^2(v) \E(\eta(v)-\E\eta(v))^2, & \Gamma_n &= \sum_{v \in V} \ell_n^3(v) \E\vert\eta(v) - \E\eta(v)\vert^3,\\
    Q_n &= \frac{\Gamma_n}{M_n^3}, & x &= \frac{2y_n^2L_{n,2}^2}{nM_n}.
\end{align*}
Since $L_{n,q}^q \geq n$ for $q=2$ and $q=3$, by \eqref{assumptions} we have $L^2_{n,2} \asymp n$ and $L^3_{n,3} \asymp n$. We thus get
\begin{align*}
    M_n^2 = L_{n,2}^2(4b^2 -4b^3\E\xi^3+ b^4(\E\xi^4-1)) \asymp y_n^2  \qquad
\text{and} \qquad \Gamma_n \asymp L_{n,3}^3 b^3 \asymp \frac{y_n^3}{n^{1/2}},
\end{align*}
which in turn implies
\begin{align*}
    Q_n \asymp \frac{1}{n^{1/2}} \qquad \text{and} \qquad x \asymp y_n \ll n^{1/2} \asymp Q_n^{-1}.
\end{align*}
It now follows from Theorem 2.1 in \cite{nagaev2002lower} that there exists positive constants $c_1$ and $c_2$ independent of $n$ such that
\begin{align*}
    \prob\left(\sum_{v \in V} \ell_n(v)(\eta(v)-\E\eta(v)) \geq \frac{2y_n^2L_{n,2}^2}{n}\right)  &\geq \left(1-\Phi\left(\frac{2y_n^2L_{n,2}^2}{nM_n}\right)\right)\left(1-c_1Q_nx\right)\exp\left(-c_2Q_nx^3\right) \\
    &=\left(1-\Phi\left(\frac{2y_n^2L_{n,2}^2}{nM_n}\right)\right)\left(1-o(1)\right) \exp\left(-o(1)\right),
    \end{align*}
since $Q_n x^3 \asymp y_n^3/n^{1/2}$ and $y_n = o(n^{1/6})$. Since $L_{n,2} \asymp n^{1/2}$ and $y_n = o(n^{1/4})$, we have
\begin{align*}
    M_n =2\frac{L_{n,2}^2y_n}{n}(1-o(1)).
\end{align*}
By Lemma \ref{silt-lemma} and that $y_n = o(n^{1/4}),$ we get $P\left(\{L_{n,2}^2\leq B_2n \} \cap \{L_{n,3}^3\leq B_3n \}\right) \sim 1$. Putting everything together, we have
\begin{align*}
       \prob \otimes P\left(\frac{T_n\sqrt{n}}{V_nL_{n,2}} \geq y_n\right) &\geq (1-o(1))(1-\Phi(y_n))P\left(\{L_{n,2}^2\leq B_2n \} \cap\{L_{n,3}^3\leq B_3n \} \right)\\
        &\sim 1-\Phi(y_n).
\end{align*}
This finishes the proof.
\section{Local time for $\Z^d$} \label{Zdlocaltimesection}
We review the necessary concentration inequalities required for the proof of Theorem \ref{Zdupperbound}. The following inequality was provided in Proposition 3.3 in \cite{asselah2019moderate}.
\begin{lemma} \label{Zdlevelset}
Define $L_n(t) = \vert \{z \in \Z^d: \ell_n(z)>t\} \vert$. There exists positive constants $c_1, c_2,c_3$ such that for $t > c_1 \log n$ and $u \geq 1$, we have
\begin{align*}
P(L_n(t)>u)  \leq c_3 \exp(-c_2\cdot t u^{1-2/d}).
\end{align*}
\end{lemma}
Based on this last lemma, we easily get the following estimate which will be needed in the proof of Theorem \ref{Zdupperbound}.
\begin{lemma} \label{Zdlocalsum}
Suppose $y_n^2 \gg \log n$ and let $t_* =   y_n^{4/(d+2)}(\log n)^{d/(d+2)}$. Then there exists a positive constant $C_1$ such that
\begin{align*}
\limsup_{n \to \infty} y_n^{-2d/(d+2)}(\log n)^{-2/(d+2)}P\left(\sum_{z: \ell_n(z)>t_*} \ell_n(z) \geq y_n^2\right) \leq -C_1.
\end{align*}
\end{lemma}
\begin{proof}
Define the sets
\begin{align*}
\mathcal D_k = \{z \in \Z^d: 2^k t_*< \ell_n(z) \leq 2^{k+1}t_* \}
\end{align*}
for $k =0,\ldots, K$, where $K$ satisfies $2^{K+1} t_* = y_n^{2d/(d+2)}(\log n)^{2/(d+2)}$. Let $a_k = \epsilon \cdot y_n^2 2^{-2k/(d-2)}$, where $\epsilon>0$ is chosen such that $\sum_{k=0}^\infty a_k \leq y_n^2$. We have
\begin{align*}
P\left(\sum_{z: \ell_n(z)> t_*} \ell_n(z) \geq y_n^2\right)  &\leq \sum_{k=0}^K P\left(\sum_{z \in \mathcal D_k} \ell_n(z)\geq a_k\right) + P\left(L_{n,\infty} > 2^{K+1} t_*\right).
\end{align*}
Reviewing the proof of Lemma \ref{max-lemma}, we see that the conclusion holds for any graph $G$ that satisfies $\inf_{v \in V}P_v(T_v^+ = \infty)> 0$. Since this applies to the graph $\Z^d$ for $d \geq 3$, we can bound the second term on the right-hand side by Lemma \ref{max-lemma}. The first term on the right-hand side is bounded by
\begin{align*}
\sum_{k=0}^K P\left(\vert \mathcal D_k \vert \geq \frac{a_k}{2^{k+1}t_*}\right) &\leq \sum_{k=0}^K \exp\left(-\left(\frac{a_k}{2}\right)^{1-2/d}(2^k t_*)^{2/d}\right)\\
&\leq (K+1)\exp(-C_2y_n^{2d/(d+2)}(\log n)^{2/(d+2)}),
\end{align*}
where the first inequality follows from Lemma \ref{Zdlevelset} and the fact that $t_* \gg \log n$ since we assume $y_n^2 \gg \log n$. We finish by observing that combinatorial factor is negligible since $K = O(\log y_n)$.
\end{proof}
As in the proof of \ref{upperbound}, we will need large deviations for the self-intersection local time $L_{n,2}^2 = \sum_{z \in \Z^d} \ell_n^2(z)$. The following result follows from Proposition 1.1 and Remark 1.3 in \cite{asselah2008large}, as well as Proposition 1.5 in  \cite{Asselah}.
\begin{lemma} \label{Zd-silt}
Let $\{S_n : n \in \N_0 \}$ be the simple random walk on $\Z^d$ and suppose $d \geq 3$. For $y > 2G(0)-1$, we have
\begin{align*}
 \limsup_{n \to \infty} n^{-1/3} \log P(L^2_{n,2} \geq y\cdot n) <0.
\end{align*}
\end{lemma}
\section{Proof of Theorem \ref{Zdupperbound}} \label{Zdupperboundproof}
We begin by defining the sets
\begin{align*}
\mathcal E^{(n)} =\mathcal E = \left\{z \in \mathcal R_n: \xi(z) \leq \frac{\sqrt{n}}{y_n \log^2 n}\right\} \qquad \text{and} \qquad
\mathcal L^{(n)} =\mathcal L= \left\{z \in \Z^d: \ell_n(z) \leq y_n^{4/(d+2)}(\log n)^{d/(d+2)}\right\}.
\end{align*}
We define the partial sums
\begin{align*}
    T_{n,1} &= \sum_{z\in \mathcal L \cap \mathcal E} \ell_n(z)\xi(z), & V^2_{n,1} &= \sum_{z\in \mathcal L \cap \mathcal E} \ell_n(z)\xi^2(z), \\
    T_{n,2} &= \sum_{z\in \mathcal L \cap \mathcal E^c} \ell_n(z)\xi(z), & V^2_{n,2} &= \sum_{z\in \mathcal L \cap \mathcal E^c} \ell_n(z)\xi^2(z)\\
    T_{n,3} &= \sum_{z\in \mathcal L^c } \ell_n(z)\xi(z), & V^2_{n,3} &= \sum_{z\in \mathcal L^c } \ell_n(z)\xi^2(z).
\end{align*}
Applying \eqref{selfintersect}, we have
\begin{align*}
&\prob \otimes P\left(T_n \sqrt{n}/(V_n L_{n,2}) \geq y_n\right) \leq \prob \otimes P\left(T_n/V_n \geq y_n\right) \\
\leq &\prob \otimes P(T_{n,1}/V_{n} \geq y_n/3)+\prob \otimes P(T_{n,2}/V_{n} \geq y_n/3) +\prob \otimes P(T_{n,3}/V_{n} \geq y_n/3).
\end{align*}
 As in the proof of \eqref{upperbound}, we apply the Cauchy-Schwartz inequality and get that this bounded by
\begin{align*}
\prob \otimes P(T_{n,1}/V_{n} \geq y_n/3)+\prob \otimes P \left(\vert \mathcal E^c \vert \geq y_n^{2d/(d+2)}\log n^{-d/(d+2)}/9\right) +P\left(\sum_{z \in \mathcal L^c} \ell_n(z) \geq y_n^2/9\right).
\end{align*}
Reviewing the proof of Lemma \ref{scenery-lemma}, we see that the conclusion holds for any random walk that satisfies $\vert \mathcal R_n \vert \leq n$. Since this holds for the simple random walk on $\Z^d$, we can apply Lemma \ref{scenery-lemma} with the assumption $\E \xi^4<\infty$, to get
\begin{align*}
\limsup_{n \to \infty } y_n^{-2d/(d+2)}(\log n)^{-2/(d+2)} \log \prob \otimes P (\vert \mathcal E^c \vert \geq y_n^{2d/(d+2)}\log n^{-d/(d+2)}/9) <0.
\end{align*}
By Lemma \ref{Zdlocalsum}, we have
\begin{align*}
\limsup_{n \to \infty} y_n^{-2d/(d+2)}(\log n)^{-2/(d+2)} \log P\left(\sum_{z \in \mathcal L^c} \ell_n(z) \geq  y_n^2/9\right) < 0.
\end{align*}
The rest of the proof is nearly identical to the proof of Lemma \ref{good-sum}. For $B > 2G(0)-1$, we have
\begin{align*}
\prob \otimes P\left( T_{n,1}/V_n\geq y_n\right) &\leq  \prob \otimes P \left( \{ V_n^2\leq  n/4 \} \cap \{ L_{n,2}^2 \leq B n \} \right)\\
 & + \prob \otimes P\left( \{ T_{n,1}/n^{1/2} \geq  y_n/2  \} \cap \{L^2_{n,2} \leq B n \}\right) + P(L_{n,2}^2 > B n) \\
&\eqqcolon  I_1 + I_2 + I_3.
\end{align*}
By Lemma \ref{Zd-silt} and that $y_n = o(n^{1/6})$, we have 
\begin{align*}
  \limsup_{n \to \infty } y_n^{-2d/(d+2)}(\log n)^{-2/(d+2)} \log I_3= -\infty,
\end{align*}
and by the same proof as in Lemma \ref{good-sum}, we have
\begin{align*}
   \limsup_{n \to \infty } y_n^{-2d/(d+2)}(\log n)^{-2/(d+2)}  \log I_1= -\infty,
\end{align*}
and so we are left to bound $I_2$. By Chebyshev's inequality, for any $\alpha_n>0$ we have
\begin{align*}
I_2 \leq \exp(-y_n \alpha_n/2)\cdot E\left[\prod_{z \in \mathcal L }\E\left[\exp\left( \ell_n (z)\xi(z)\frac{\alpha_n}{n^{1/2}}  \1\{z \in \mathcal E\}\right)\right] \1\{L_{n,2}^2  \leq B_2 n\}\right].
\end{align*}
We set $\alpha_n  = y_n^{(d-2)/(d+2)} (\log n)^{2/(d+2)}$, and observe that there exists $C>0$ such that for $z \in \mathcal L $, we have
\begin{align*}
    \ell_n(z) \xi(z)\frac{\alpha_n}{n^{1/2}}  \1\{z \in \mathcal E\} \leq C/\log n.
\end{align*}
There exists a constant $C>0$ such that for $x\leq 1$ we have $e^x \leq 1+x+C x^2$. Combining both facts yields the inequality
\begin{align*}
 \E \left[ \exp\left(\ell_n(z)\xi(z)\frac{\alpha_n}{n^{1/2}}  \1\{z \in \mathcal E\}\right) \right] &\leq  \E \left[1+ \ell_n(z)\xi(z)\frac{\alpha_{n}}{n^{1/2}}  \1\{z \in \mathcal E\}+ C\frac{\ell^2_n(z)}{n} \xi^2(z) \alpha_n^2 \1\{z \in \mathcal E\} \right]\\
  &\leq \exp\left( C  \frac{\ell_n^2(z)}{n}\alpha_n^2\right).
 \end{align*}
 We thus have
\begin{align*}
   I_2&\leq \exp(- y_n^{2d/(d+2)}(\log n)^{2/(d+2)}/2+ C B_2 y_n^{(2d-4)/(d+2)}(\log n)^{4/(d+2)}).
\end{align*}
Observe that $y_n^{2d/(d+2)}(\log n)^{2/(d+2)} \gg y_n^{(2d-4)/(d+2)}(\log n)^{4/(d+2)}$ precisely when $y^2_n \gg \log n$, which completes the proof.

\newpage

\noindent{\bf Acknowledgements.}
I would like to thank my M.Sc. advisor Ofer Zeitouni for his guidance and support throughout all stages of this work. This project has received funding from the European Research Council (ERC) under the European Union’s Horizon 2020 research and innovation programme (grant agreement No. 692452). Thank you to both reviewers for providing detailed feedback that greatly improved the content and flow of the manuscript.

\bibliographystyle{acm}
\bibliography{references}

\begin{thebibliography}{10}

\bibitem{asselah2008large}
{\sc Asselah, A.}
\newblock Large deviations estimates for self-intersection local times for
  simple random walk in $\mathbb{Z}^3$.
\newblock {\em Probability theory and related fields 141}, 1-2 (2008), 19--45.

\bibitem{Asselahnote}
{\sc Asselah, A., and Castell, F.}
\newblock A note on random walk in random scenery.
\newblock {\em Annales de l'I.H.P. Probabilit\'es et statistiques 43}, 2
  (2007), 163--173.

\bibitem{Asselah}
{\sc Asselah, A., and Castell, F.}
\newblock {Random walk in random scenery and self-intersection local times in
  dimensions d $\geq$ 5}.
\newblock {\em Probability Theory and Related Fields 138}, 1-2 (2 2007), 1--32.

\bibitem{asselah2019moderate}
{\sc Asselah, A., and Schapira, B.}
\newblock Moderate deviations for the range of a transient random walk. {II}.
\newblock arXiv:1909.01925, 2019.

\bibitem{Borodin1982}
{\sc Borodin, A.~N.}
\newblock Limit theorems for sums of independent random variables defined on a
  nonrecurrent random walk.
\newblock {\em Journal of Soviet Mathematics 20}, 3 (Oct 1982), 2130--2137.

\bibitem{Dembo2002}
{\sc Dembo, A., Gantert, N., Peres, Y., and Zeitouni, O.}
\newblock Large deviations for random walks on galton--watson trees: averaging
  and uncertainty.
\newblock {\em Probability Theory and Related Fields 122}, 2 (Feb 2002),
  241--288.

\bibitem{Feng}
{\sc Feng, X., Shao, Q.-M., and Zeitouni, O.}
\newblock Self-normalized moderate deviations for random walk in random
  scenery.
\newblock {\em Journal of Theoretical Probability\/} (2019), 1--22.

\bibitem{FLEISCHMANN2008}
{\sc Fleischmann, K., M{\"o}rters, P., and Wachtel, V.}
\newblock Moderate deviations for a random walk in random scenery.
\newblock {\em Stochastic Processes and their Applications 118}, 10 (2008),
  1768 -- 1802.

\bibitem{gantert2006deviations}
{\sc Gantert, N., van~der Hofstad, R., and K{\"o}nig, W.}
\newblock Deviations of a random walk in a random scenery with stretched
  exponential tails.
\newblock {\em Stochastic processes and their applications 116}, 3 (2006),
  480--492.

\bibitem{Kesten1979}
{\sc Kesten, H., and Spitzer, F.}
\newblock A limit theorem related to a new class of self similar processes.
\newblock {\em Zeitschrift f{\"u}r Wahrscheinlichkeitstheorie und Verwandte
  Gebiete 50}, 1 (Jan 1979), 5--25.

\bibitem{Lyons1996}
{\sc Lyons, R., Pemantle, R., and Peres, Y.}
\newblock Biased random walks on galton--watson trees.
\newblock {\em Probability Theory and Related Fields 106}, 2 (Oct 1996),
  249--264.

\bibitem{nagaev2002lower}
{\sc Nagaev, S.~V.}
\newblock Lower bounds on large deviation probabilities for sums of independent
  random variables.
\newblock {\em Theory of Probability and Its Applications 46}, 1 (2002),
  79--102.

\bibitem{shao1997self}
{\sc Shao, Q.-M.}
\newblock Self-normalized large deviations.
\newblock {\em The Annals of Probability 25}, 1 (1997), 285--328.

\end{thebibliography}

\end{document}